\documentclass[12pt]{amsart}
\usepackage[dvips]{epsfig}
\usepackage{graphics}
\usepackage{latexsym}
\usepackage{verbatim}
\usepackage{amsmath}
\usepackage{amsthm}
\usepackage{amssymb}
\usepackage{color}
\newtheorem{theorem}{Theorem}
\newtheorem{lemma}[theorem]{Lemma}

\newtheorem{remark}[theorem]{Remark}

\newcommand{\ud}{{\rm d}}

\let\cd\cdot

\def\signcs{\bigskip\hspace{80mm}
\vbox{{\sc C. Schmeiser\par\vspace{2mm}
Universit\"{a}t Wien\par
Oskar-Morgenstern-Platz 1\par
1090 Wien\par
Austria \par\vspace{1mm}
e-mail:} Christian.Schmeiser@univie.ac.at  }}
\def\signln{\bigskip\hspace{80mm}
\vbox{{\sc L. Neumann\par\vspace{2mm}
Universit\"{a}t Innsbruck\par
Technikerstra\ss e 13\par
6020 Innsbruck\par
Austria \par\vspace{1mm}
e-mail:} {Lukas.Neumann@uibk.ac.at} }}

\begin{document}

\title[Kinetic reaction model]
{A kinetic reaction model: decay to equilibrium and macroscopic limit}

\author{Lukas Neumann, Christian Schmeiser}

\hyphenation{bounda-ry rea-so-na-ble be-ha-vior pro-per-ties
cha-rac-te-ris-tic re-com-bi-na-tion}

\begin{abstract} 
We propose a kinetic relaxation-model to describe a generation-recom\-bination reaction of two species. 
The decay to equilibrium is studied by two recent methods \cite{DMS,MouNeu} for proving hypocoercivity of the linearized equations. 
Exponential decay of small perturbations can be shown for the full nonlinear problem.
The macroscopic/fast-reaction limit is derived rigorously employing entropy decay, resulting in a nonlinear diffusion equation for the difference 
of the position densities. 
\end{abstract}

\maketitle




\section{Introduction}
\subsection{The model}
\setcounter{equation}{0}
We consider the system
\begin{equation}\label{nonlin}
\begin{split}
&\partial_t f+v\cd\nabla_xf=\chi_1(v)-\rho_g f \,,\\
&\partial_t g+v\cd\nabla_xg=\chi_2(v)-\rho_f g\,,
\end{split}
\end{equation}
where $f$ and $g$ depend on position $x\in\mathbb{T}^3$, the three dimensional torus of volume one, on velocity $v\in\mathbb{R}^3$, and on
time $t\ge 0$. They represent the phase space densities of chemical reactants A and B, which are produced (with nonnegative velocity profiles $\chi_1$ and $\chi_2$, respectively) by the decomposition of a substance C. The density of the substance C is not subject of our study and is assumed to be fixed. 
On the other hand the substances A and B can recombine to form C and thus be eliminated from our system. Similar models have been used
for generation and recombination of electron-hole pairs in semiconductors \cite{CDS,DNS}.

 The probability of the reaction is depending on the position density 
\begin{equation*}
\rho_h(x,t) := \int_{\mathbb{R}^3} h(x,v,t) \,\ud v\,,
\end{equation*}
of the reaction partner. We consider the system \eqref{nonlin} subject to initial conditions
\begin{equation} \label{IC}
  f(x,v,0) = f_0(x,v) \,,\quad g(x,v,0) = g_0(x,v) \,,
\end{equation}
with nonnegative initial data $f_0$ and $g_0$. 
Moreover, since we want to describe the reaction 1A+1B$\leftrightarrow$1C, we require that $\int_{\mathbb{R}^3}(\chi_1-\chi_2) \ud v=0$.
We assume that the system has been nondimensionalized and that $\chi_1$ and $\chi_2$ satisfy 
\begin{align} 
 & (1+|v|^2)\chi_j \in L^1(\mathbb{R}^3) \cap L^\infty(\mathbb{R}^3) \,,\qquad \chi_j > 0 \,,\notag\\
 &\int_{\mathbb{R}^3}\chi_j\ud v=1 \,,\qquad 
  \int_{\mathbb{R}^3} v\chi_j\ud v= 0\,, \label{ass:chi}\\
 & \exists C,\theta>0:\quad \forall a\in\mathbb{R}\,,\omega\in \mathbb{S}^2\,, \delta>0 :\quad 
  \int_{|a+v\cdot\omega|<\delta}\chi_j dv \le C\delta^\theta\,,\quad j=1,2\,.\notag
\end{align}
The last line will be needed for an $L^2$ averaging lemma with the weight $1/\chi_j$. The largest value 
to be expected for the exponent is $\theta=1$, which is achieved, e.g., for Gaussian distributions, the prototypical examples for the $\chi_j$,
but more generally also for $\chi_j(v) \le c(1+|v|^2)^{-k}$ with $k>1$.

Note that, at least formally, the mass difference is conserved:
\[
 \frac{\ud}{\ud t}\int_{\mathbb{T}^3} \int_{\mathbb{R}^3}(f-g)\,\ud v\,\ud x=0\,,
\]
as can be seen by subtraction of the two equations and subsequent integration. This is to be expected since by the reaction molecules of A and B are created 
and destroyed pairwise. We introduce the unique $\rho_\infty> 0$, such that 
\begin{equation}\label{rhoinf}
 \int_{\mathbb{T}^3} \int_{\mathbb{R}^3} (f_0-g_0)\ud v\ud x= |\mathbb{T}^3| \left(\rho_\infty-\frac{1}{\rho_{\infty}} \right) \,,
\end{equation}
and expect convergence as $t\to\infty$ of solutions of \eqref{nonlin}, \eqref{IC} to the steady state 
\begin{equation*}
f_\infty(x,v)=\rho_\infty\chi_1(v)\,,\qquad  g_\infty(x,v)=\frac{1}{\rho_\infty}\chi_2(v) \,. 
\end{equation*}
This is supported by the decay properties of the entropy functional
\begin{equation}\label{entropy}
H(f,g)=\int f\left(\ln\frac{f}{\rho_\infty\chi_1}-1\right)+ g\left(\ln\frac{\rho_\infty g}{\chi_2}-1\right)\,\ud v\,\ud x\ ,
\end{equation}
which decreases as long as $(f,g)$ is different from $(\rho(x)\chi_1(v),\chi_2(v)/\rho(x))$ for some $\rho(x)$:
\begin{equation}\label{log-entropy}
\frac{\ud}{\ud t}H(f,g)=\int\int\int(\chi_1\chi'_2-fg')\ln\frac{fg'}{\chi_1\chi'_2}\,\ud v'\,\ud v\,\ud x\leq0\ .
\end{equation}
Among these functions $(f_\infty,g_\infty)$ is the only solution of \eqref{nonlin} with the same mass difference as the initial data.

Spectral stability of the equilibrium will be investigated by linearization:
\begin{equation}\label{lin2}
\begin{split}
\partial_t f+v\cd\nabla_xf&=-\rho_\infty\chi_1\rho_g- \frac{1}{\rho_\infty}f \,,\\
\partial_t g+v\cd\nabla_xg&=-\frac{1}{\rho_\infty}\chi_2\rho_f-\rho_\infty g\,,
\end{split}
\end{equation}
where for simplicity the perturbations have again been denoted by $f$ and $g$, now satisfying
\begin{equation}\label{mass-zero} 
  \int_{\mathbb{T}^3} \int_{\mathbb{R}^3} (f-g)dv\,dx = \int_{\mathbb{T}^3} (\rho_f - \rho_g)dx = 0 \,.
\end{equation}

Rigorous results for kinetic equations for chemically reacting species with nonlinear reaction models are scarce in the literature. An example
is an existence result \cite{Pol} for a model with quadratic reaction terms under rather weak natural assumptions on the initial data, 
where stability is based on entropy decay. Since existence of solutions is not the main focus of this work, we shall make rather strong 
assumptions on the data, with the consequence that the existence and uniqueness proof in the following section is based on weighted $L^\infty$ 
estimates and rather straightforward.

The analysis of the decay to equilibrium is complicated by the fact that the entropy dissipation \eqref{log-entropy} vanishes not only
for the equilibrium, but on a larger set of {\em local equilibria.} If exponential decay to equilibrium can still be proven, the system is called
{\em hypocoercive} \cite{Vil}. The authors have been involved in the development of two different abstract procedures for proving hypocoercivity
for linear equations \cite{DMS,MouNeu}, both based on the construction of suitable Lyapunov functionals (or {\em modified entropies}), whose
dissipation functionals have appropriate coercivity properties. The method of \cite{DMS} is based on a slightly tilted, weighted $L^2$-norm,
while \cite{MouNeu} works in a $H^1$ setting and can be extended to higher regularity. In Section 2 we show that both methods are applicable 
to a linearized version of \eqref{nonlin}.
Since the estimates of the existence result already provide neutral stability of the equilibrium, the decay results can be extended to
a local asymptotic stability result with exponential convergence for the full nonlinear model. The decay rates proven by both methods can,
in principle, be computed explicitly. Complete formulas would however be rather complicated, whence we did not attempt a comparison.
An essential difference between the methods is the weaker assumptions on initial data in \cite{DMS}. On the other hand, the method of 
\cite{MouNeu} has the potential to provide strong convergence properties including derivatives.

In Section \ref{formal-limit} the macroscopic/fast-reaction limit is carried out formally, leading to a nonlinear diffusion equation for the 
difference of the position densities of the reactants. Similar results have been derived for reaction-diffusion systems \cite{Hilhorst1,Hilhorst2}
and for coagulation-fragmentation models \cite{CDF1,CDF2}. A rigorous justification of the limit is the subject of Section 3. It is based on 
an analysis of the entropy dissipation functional \eqref{log-entropy} and adapts the procedure of \cite{PouSch}, where compactness is obtained
from an averaging lemma in weighted $L^2$-spaces. We prove a slightly generalized version compared to \cite{PouSch}.

We note that with the torus we chose the simplest geometric setting. Natural modifications include bounded domains with specular reflection
boundary conditions or whole space problems with confining potentials. We conjecture that our results can be extended to these situations,
however with considerably more technical effort for the latter (see, e.g., the hypocoercivity results with confining potentials in \cite{DMS}).

\subsection{Existence of solutions}

The entropy decay relation \eqref{log-entropy} would suggest an existence result for initial data with bounded entropy. Such a result for a 
similar problem has been proven in \cite{Pol}. The main ingredients are entropy inequalities, weak $L^1$ compactness and velocity averaging. 
These ideas might be transferable to our situation. However, for our purposes we need more information on the solutions. Under stronger
assumptions on the initial data, a global existence result can be proved easily.

\begin{theorem} \label{existence}
Let \eqref{ass:chi} hold and let there be positive constants $\gamma_1 < \rho_\infty$  and $\gamma_2$ such that the initial data 
$f_0$, $g_0\in L^\infty(\mathbb{T}^3\times \mathbb{R}^3)$ satisfy
\[ 
(\rho_\infty - \gamma_1)\chi_1\leq f_0\leq (\rho_\infty + \gamma_2) \chi_1\quad\text{and}\quad 
\tfrac{1}{\rho_\infty + \gamma_2}\chi_2\leq g_0\leq \tfrac{1}{\rho_\infty - \gamma_1} \chi_2 \,.
\]
Then the initial value problem \eqref{nonlin}, \eqref{IC} has a unique global mild solution 
$(f,g) \in C([0,\infty), L^\infty(\mathbb{T}^3\times\mathbb{R}^3))^2$ satisfying
\begin{equation*}
(\rho_\infty - \gamma_1)\chi_1(v)\leq f(x,v,t)\leq  (\rho_\infty + \gamma_2) \chi_1(v) \,,\qquad 
    (x,v,t)\in \mathbb{T}^3 \times\mathbb{R}^3 \times [0,\infty) \,,
\end{equation*}
and
\begin{equation*}
  \tfrac{1}{\rho_\infty + \gamma_2}\chi_2(v)\leq g(x,v,t)\leq \tfrac{1}{\rho_\infty - \gamma_1} \chi_2(v) \,,\qquad 
    (x,v,t)\in \mathbb{T}^3 \times\mathbb{R}^3 \times [0,\infty)\,. 
\end{equation*}
\end{theorem}
\begin{proof}
The mild formulation of the equation for $f$ is given by
\begin{align*}
  f(x,v,t) =\,& f_0(x-vt)\exp\left( -\int_0^t \rho_g(x+v(s-t),s)ds\right) \\
   & + \chi_1(v) \int_0^t \exp\left( -\int_s^t \rho_g(x+v(\tau-t),\tau)d\tau\right)ds \,,
\end{align*}
and an analogous equation holds for $g$. It is easily seen that the set of $(f,g)$ defined by the estimates of the theorem is mapped into
itself by the right hand sides. This provides the a priori estimate needed for the continuation of a local solution constructed by Picard iteration.
\end{proof}

\subsection{Formal macroscopic limit}\label{formal-limit}

In this section we formally derive a macroscopic limit of \eqref{nonlin}. The limit will be made rigorous in Section~\ref{RigHy}.
Since by \eqref{ass:chi} the mean velocities of the equilibrium distributions vanish, we adopt a diffusive (or parabolic) scaling 
$t\rightarrow t/\epsilon^2$ and $x\rightarrow x/\epsilon$ and derive
\begin{equation}\label{rescaled}
\begin{split}
&\epsilon^2\partial_t f+\epsilon v\cd\nabla_xf=\chi_1(v)-\rho_g f \\
&\epsilon^2\partial_t g+\epsilon v\cd\nabla_xg=\chi_2(v)-\rho_f g\ .
\end{split}
\end{equation}
We substitute the ansatz
\begin{equation*}
f=f^0+\epsilon f^1+O(\epsilon^2) \text{ and }g=g^0+\epsilon g^1+O(\epsilon^2)\ .
\end{equation*}
Balancing the leading order terms gives $\rho_{g^0}f^0=\chi_1$ and $\rho_{f^0}g^0=\chi_2$. 
This is equivalent to the existence of $\rho(x,t)$ such that
\begin{equation*}
f^0(x,v,t)=\rho(x,t)\chi_1(v) \quad\text{ and }\quad g^0(x,v,t)=\frac{1}{\rho(x,t)}\chi_2(v) \,.
\end{equation*}
Now we balance the first order terms in $\epsilon$ and derive
\begin{equation*}
\begin{split}
v\cd\nabla_x f^0&=-\rho_{g^1}f^0-\frac{1}{\rho}f^1\\
v\cd\nabla_x g^0&=-\rho_{f^1}g^0-\rho g^1\ .
\end{split}
\end{equation*}
Due to \eqref{ass:chi} the solvability condition $\int_{\mathbb{R}^3} v\cd \nabla_x(f^0 - g^0)dv = 0$ is satisfied, and we obtain
\begin{equation*}
\begin{split}
f^1&=-\rho \chi_1 v\cd\nabla_x \rho+\rho^1\chi_1 \,,\\
g^1&=\frac{1}{\rho^3}\chi_2 v\cd \nabla_x \rho - \frac{\rho^1\chi_2}{\rho^2} \,,
\end{split}
\end{equation*}
where the second terms on the right hand side constitute the general solution of the homogeneous problem.
Now we substitute this into the limit of the conservation equation
\[
  \partial_t (\rho_f-\rho_g) + \nabla_x\cdot\left( \frac{1}{\epsilon} \int_{\mathbb{R}^3}v(f-g)\ud v\right) = 0 \,,
\]
to obtain
\begin{equation*}
\partial_t\left(\rho-\frac{1}{\rho}\right)=\nabla_x\cdot\left[\left(D_1\rho+\frac{D_2}{\rho^3}\right)\nabla_x\rho\right]\ ,
\end{equation*}
where we have introduced the positive definite symmetric matrices
\begin{equation*}
D_1=\int v\otimes v\chi_1 \,\ud v\ \text{ and }\ D_2=\int v\otimes v\chi_2 \,\ud v\ .
\end{equation*}
This can be written as the nonlinear diffusion equation 
\begin{equation*}
\partial_t m=\nabla_x \cdot\left(D(m)\nabla_x m\right)\,,
\end{equation*}
for the new unknown
\begin{equation*}
 m =  \rho - \frac{1}{\rho}\,,\qquad\mbox{i.e. } \rho(m) = \frac{1}{2}\left(m + \sqrt{m^2+4}\,\right) \,,
\end{equation*}
where we have introduced the diffusion matrix 
\begin{equation*}
D(m)=\left(D_1\rho(m)^2+\frac{D_2}{\rho(m)^2}\right)\frac{1}{2\rho(m)-m} \,.
\end{equation*}
The unknown $m$ is the zeroth order approximation of the difference of the position densities of $f$ and $g$.

\section{Long time properties}
In this section we study decay to equilibrium for solutions of \eqref{nonlin} and of the linearized problem \eqref{lin2}, \eqref{mass-zero}.
In order to estimate the decay towards the equilibrium quantitatively we introduce the $L^2$ scalar product, weighted with the steady state measure,
\begin{equation}\label{def:scalar}
 \left\langle F_1,F_2 \right\rangle=\int_{\mathbb{T}^3} \int_{\mathbb{R}^3} 
 \left( \frac{f_1 f_2}{\rho_\infty \chi_1}+\frac{g_1 g_2\rho_\infty}{\chi_2}\right) \ud v\, \ud x \,,\quad\mbox{with } F_j = \binom{f_j}{g_j}\,.
\end{equation}
Throughout this article we denote by $\|\cdot\|$ the norm induced by this scalar product.
The orthogonal projection onto the null space of the linearized collision operator $\mathsf{L}$, 
\begin{equation} \label{def:L}
 \mathsf{L}F=\binom{-\rho_\infty\chi_1\rho_g-\frac{1}{\rho_\infty}f}{-\frac{1}{\rho_\infty}\chi_2\rho_f-\rho_\infty g}\,,
   \quad\mbox{with } F = \binom{f}{g}\,,
\end{equation}
that is the space of local equilibria, is given by
\begin{equation}\label{localequi}
\Pi F := \frac{\rho_f - \rho_g}{\rho_\infty^2 + 1}
\begin{pmatrix}  \rho_\infty^2 \chi_1\\  - \chi_2 \end{pmatrix}\,.
\end{equation}
Straightforward computations yield
\begin{equation*}
 -\left\langle\mathsf{L}F,F\right\rangle 
=\int_{\mathbb{T}^3} \left(\int_{\mathbb{R}^3} \left(\frac{(f-\rho_f \chi_1)^2}{\chi_1\rho_\infty^2}+\frac{(g-\rho_g \chi_2)^2\rho_\infty^2}{\chi_2}\right)\mathrm{d} v+\left(\frac{\rho_f}{\rho_\infty} +\rho_g\rho_\infty\right)^2 \right)\mathrm{d} x \,,
\end{equation*}
and
\begin{align*}
 &\left\|(1-\Pi)F\right\|^2 \\ &=\int_{\mathbb{T}^3} \left(\int_{\mathbb{R}^3} \left(\frac{(f-\rho_f \chi_1)^2}{\chi_1\rho_\infty}+\frac{(g-\rho_g \chi_2)^2\rho_\infty}{\chi_2}\right)\ud v + \frac{\rho_\infty}{\rho_\infty^2+1}\left(\frac{\rho_f}{\rho_\infty} +\rho_g\rho_\infty\right)^2 \right)\ud x \,,
\end{align*}
implying the {\em microscopic coercivity} estimate
\begin{equation}\label{microcoerc}\tag{P1, H3}
-\left\langle\mathsf{L}F,F\right\rangle\geq \min\{\rho_\infty,1/\rho_\infty\} \left\|(1-\Pi)F\right\|^2\,.
\end{equation}
This gives a quantitative estimate of the decay towards the local equilibrium introduced by the linearized collision operator.
In the spatially homogeneous situation such an estimate is enough to prove exponential decay to equilibrium for the linearized equation.

For spatially non homogeneous situations we expect the densities to become constant as these are the only local equilibria that also annihilate the transport part of the equation. 
The complete relaxation mechanism can be seen as a combination of local relaxation in the 
velocity direction by the collision operator and an interplay between mixing by the transport operator and confinement in our bounded spatial domain.
In the following we will study two recent methods (\cite{DMS,DMS-CRAS} and \cite{MouNeu}) to estimate the decay in the spatially non-homogeneous situation. 
Both rely on properties of the linearized collision operator that we will verify in the sequel. Concerning their numbering, (Hn) and (Pm) in this work correspond to (Hn) in \cite{MouNeu} and, respectively, (Hm) in \cite{DMS}.
The microscopic coercivity property \eqref{microcoerc} is needed in both approaches.

\subsection{Coercivity in weighted $H^1$}\mbox{}\\
When studying coercivity of the collision operator we saw that the operator provides 
coercivity only with respect to the velocity distribution. The strategy in \cite{MouNeu} is to transfer some of this dissipation 
effect in the velocity to the spacial variable by using the mixing effect of the transport operator. 
This method was mainly inspired by discussions with C. Villani (see \cite{Vil}) and results by Y. Guo (see for example \cite{Guo:PB:02} or the references in \cite{MouNeu}).\\
We start by writing the linearized equation \eqref{lin2} in the abstract form 
\begin{equation}\label{eq:lin-abstr}
  \frac{\ud F}{\ud t} + \mathsf{T} F = \mathsf{L} F \,, 
\end{equation}  
with $F=(f,g)^T$, with the transport operator
\[
 \mathsf{T} F = \mathsf{T}\binom{f}{g}:=\binom{v\cdot\nabla_x f}{v\cdot\nabla_x g}\,,
\]
and with the linearized collision operator $\mathsf{L}$ given in \eqref{def:L}.
The regularizing effect of the transport operator can be quantified by looking at the time evolution of mixed derivative terms 
\begin{equation}\label{mixedder}
 \tfrac{\ud }{\ud t}\langle\nabla_v F,\nabla_x F\rangle=-\langle\nabla_x F,\nabla_v \mathsf{T} F\rangle+2\langle\nabla_v\mathsf{L}F,\nabla_x F\rangle=-\|\nabla_x F\|^2+rest\, ,
\end{equation}
where the derivatives are to be understood in the sense $\nabla F=(\nabla f,\nabla g)^T$.
Thus the transport operator provides some damping in $x$-derivatives provided the rest terms can be compensated by boundedness and coercivity in the velocity direction of the collision operator.
We introduce $H^1$, the Sobolev--spaces with weights $\chi_1$ and $\chi_2$ and the norm
\begin{equation*}
\left\|F\right\|_{H^1}^2= \left\| F\right\|^2+\left\| \nabla_x F\right\|^2+\left\|\nabla_v F\right\|^2\,.
\end{equation*}

In \cite{MouNeu} a norm $\|F\|_{\mathcal{H}^1}^2=\|F\|_{H^1}^2+\delta\langle\nabla_x F,\nabla_v F\rangle$ is constructed that is, for $\delta$ small enough, 
equivalent to the $H^1$--norm but uses these mixed derivatives to get a coercivity estimate also for the spatially non-homogeneous situation.
Exponential convergence in $\mathcal{H}^1$ is then derived from
\begin{equation}\label{coercH1Estim}
 \tfrac{\ud }{\ud t}\|F\|_{\mathcal{H}^1}^2\leq -\tau\|F\|_{\mathcal{H}^1}^2\,,
\end{equation}
and convergence in $H^1$ with the correct weight follows from the equivalence of the norms.
Since the essential mixing effect of the transport operator is quantified by means of \eqref{mixedder} this approach is suitable only for proving coercivity in a space of higher differentiability, 
more precisely $H^1$ spaces, weighted with the the steady state (or Gibbs-) measure. 
This is a restriction of the method on one hand but on the other hand the method can also be employed to prove convergence in higher derivatives. 
This allows to use embedding theorems and lends itself to studying nonlinear equations in the perturbative regime. 
In the model we study here, however, we have good a priory bounds already and thus there is no need to go beyond $H^1$ as we will see in the proof of Theorem~\ref{H1nonlin}.\\

To control remainder terms in the time evolution of the $\mathcal{H}^1$-norm structural assumptions on the linearized collision operator are made. We will verify these in the sequel. 
For a collision operator that is \emph{local} in $x$, as is the case for our model, the spatial derivatives cannot increase as can be seen by using $\eqref{microcoerc}$ on the derivatives. 
To get decay of the velocity derivatives it is essential that the reaction operators can be split $\mathsf{L}=K-\Lambda$, into a ``loss'' part 
\[
\Lambda\binom{f}{g}=\binom{\tfrac{1}{\rho_\infty} f}{\rho_\infty g}
\]
that is (in our case trivially) coercive on the $L^2$--level 
\begin{equation}\label{losscoerc}
\left\langle\Lambda F,F\right\rangle\geq \min\left\{\rho_\infty,\tfrac{1}{\rho_\infty}\right\}\left\|F\right\|^2\tag{H1/2}\,,
\end{equation}
and a gain part
\[
 K\binom{f}{g}=\binom{-\rho_\infty\chi_1\rho_g}{-\tfrac{1}{\rho_\infty}\chi_2\rho_f}\,,
\]
 which is regularizing in $v$ as long as $\chi_1$ and $\chi_2$ are regular. Indeed, for $\nabla_v$ acting component-wise in the two functions, a standard interpolation argument yields
\begin{equation}\label{gainreg}
\forall\delta > 0\colon \exists C \,\colon\,\left| \left\langle\nabla_v K F,\nabla_vF \right\rangle\right|\leq
\delta\left\|\nabla_vF\right\|^2+C\left\|F\right\|^2\tag{H2}\,,
\end{equation}
where $C$ depends on $\|\nabla_v\chi_1\|_{L^\infty}$ and $\|\nabla_v\chi_2\|_{L^\infty}$.\\
Since the method relies on $H^1$ type estimates actually coercivity in $H^1$ of the loss part is necessary. In our case the same estimate as in \eqref{losscoerc} results in 
\begin{equation}\label{losscoerch1}
 \left\langle\nabla_v\Lambda F,\nabla_v F\right\rangle\geq\min\left\{\rho_\infty,\tfrac{1}{\rho_\infty}\right\} \left\|\nabla_v F\right\|^2\tag{H1/3}\,
\end{equation}
where in general negative terms of lower order derivatives are allowed but not needed in this case.

Control of terms of the type $\langle\nabla \mathsf{L}F,\nabla F\rangle$ is ensured by the fact that the linearized collision operator is bounded in the sense that
\begin{multline}\label{lincolbd}
\left|\left\langle\mathsf{L}\binom{f}{g},\binom{u}{v}\right\rangle\right| \\
\leq\int_{\mathbb{R}^3}\int_{\mathbb{T}^3}\left|\rho_g u+\frac{fu}{\rho_\infty^2\chi_1}+\rho_fv+\frac{gv\rho_\infty^2}{\chi_2}\right| \ud v\,\ud x \\
\leq4\max\{\rho_\infty,\tfrac{1}{\rho_\infty}\} \left\|\binom{f}{g}\right\|\left\|\binom{u}{v}\right\|\,.\tag{H1/1}
\end{multline}
This can be seen by using $\int\chi_1\ud v=\int\chi_2\ud v=1$ and Cauchy Schwarz in estimates of the type
\begin{multline*}
\int_{\mathbb{R}^3}\int_{\mathbb{T}^3}\rho_g u\,\ud v\,\ud x\leq 
\sqrt{\int_{\mathbb{R}^3}\int_{\mathbb{T}^3}\rho_\infty\chi_1\rho_g^2\ud v\,\ud x} \sqrt{\int_{\mathbb{R}^3}\int_{\mathbb{T}^3}\frac{u^2}{\rho_\infty \chi_1}\ud v\,\ud x} \\
=\sqrt{\int_{\mathbb{R}^3}\rho_\infty\rho_g^2\,\ud x} \sqrt{\int_{\mathbb{R}^3}\int_{\mathbb{T}^3}\frac{u^2}{\rho_\infty \chi_1}\ud v\,\ud x}\leq 
\sqrt{\int_{\mathbb{R}^3}\int_{\mathbb{T}^3}\frac{g^2\rho_\infty}{\chi_2}\ud v\,\ud x} \sqrt{\int_{\mathbb{R}^3}\int_{\mathbb{T}^3}\frac{u^2}{\rho_\infty \chi_1}\ud v\,\ud x}\,.
\end{multline*}

Properties $\eqref{lincolbd}$ to $\eqref{losscoerch1}$ together with $\eqref{gainreg},\, \eqref{microcoerc}$ ensure that we can use the main theorem from \cite{MouNeu} to derive the following convergence result for the linearized problem:
\begin{theorem}
Let $\chi_1$ and $\chi_2$ be in $W^{1,\infty}\left(\mathbb{R}^3\right)$ and the initial data $f_0$ be in $H^1(\ud v/\chi_1)$ and $g_0$ in $H^1(\ud v/\chi_2)$.
Then the solutions $f,g$ of the linearized problem \eqref{lin2} subject to initial conditions $f(t=0)=f_0$, $g(t=0)=g_0$ exist globally and converges exponentially to the equilibrium distribution. 
For $\int_{\mathbb{T}^3}\int_{\mathbb{R}^3}(f_0-g_0)\ud v\,\ud x = 0$ the equilibrium is zero and we have
\begin{equation*}
\|f(\cdot,\cdot,t)\|_{H^1(\ud v/\chi_1)} + \|g(\cdot,\cdot,t)\|_{H^1(\ud v/\chi_2)}  \leq C\exp(-\tau t)\,, 
\end{equation*}
where the rate $\tau$ depends on the constants in the estimates tagged with $\mathrm{(H1)}$--\,$\mathrm{(H3)}$ and the constant $C$ also depends on the norm of the initial data in $H^1$ and of $\chi_i$ in $W^{1,\infty}$.
\end{theorem}
\begin{remark}Convergence in higher order Sobolev--spaces can be derived straightforwardly provided the estimates \eqref{losscoerch1} and \eqref{gainreg} can be generalized to higher order derivatives, 
as is easily verified to be the case for our model.
\end{remark}
This feature is usefull mainly in applying the results to the nonlinear system in a perturbative setting. 
Control of the bilinear contribution in the interaction is given by applying chain rule, H{\"o}lder inequality and using Sobolev--embedding to lower the exponents in the norm to two again. In Dimension $3$ we see that 
\begin{equation}\label{bilinbd}\tag{H4}
 \forall k\geq2\colon \left\|\binom{\rho_gf}{\rho_fg}\right\|_{H^k}\leq C\left\|\binom{f}{g}\right\|_{H^k}^2\,.
\end{equation}
Now the exponential decay of the $H^2$-- norm in the linearized situation can be used to conclude that for initial data close to the global equilibrium in $H^2$ we have convergence to the stationary state.

Here however we want to give a stronger result -- in the sense that less regularity is necessary -- by using the a priori bounds of Theorem~\ref{existence}.
\begin{theorem}\label{H1nonlin}
Let $\chi_1$ and $\chi_2$ be in $W^{1,\infty}\left(\mathbb{R}^3\right)$ and the initial data $f_0$ be in $H^1(\ud v/\chi_1)$ and $g_0$ in $H^1(\ud v/\chi_2)$.
Moreover let $f_0$ and $g_0$ satisfy the assumptions of Theorem~\ref{existence} with $\gamma_1$ and $\gamma_2$ small enough.
Then the solution to equations~\eqref{nonlin} with initial data $f_0$, $g_0$ exists globally in time and converges to the unique stationary state exponentially, more precisely
\begin{equation*}
  \|f(\cdot,\cdot,t)-\rho_\infty \chi_1\|_{H^1(\ud v/\chi_1)} + \|g(\cdot,\cdot,t)-\chi_2/\rho_\infty\|_{H^1(\ud v/\chi_2)}  \le C e^{-\tau t} \,,
\end{equation*}
where the constants $\tau$ depends only on the estimates $\mathrm{(H1)}$--\,$\mathrm{(H3)}$ and $C$ depends also on $\gamma_1$, $\gamma_2$ as well as the $W^{1,\infty}$-norms of $\chi_i$.
\end{theorem}
\begin{proof}
We use the results for the linearized collision operator and regard the difference to the nonlinear one as a small perturbation. This difference is given by
\[
  Q(f,g) - \mathsf{L} F = - \binom{(\rho_g - 1/\rho_\infty)(f-\rho_\infty \chi_1)}{(\rho_f - \rho_\infty)(g-\chi_2/\rho_\infty)} \,,
\]
and, by Theorem \ref{existence}, $\rho_f$ is close to $\rho_\infty$ and $\rho_g$ close to $1/\rho_\infty$, yielding
\begin{equation}\label{quadrpert}
  \| Q(f,g) - \mathsf{L} F\| \le \gamma \|F\|
\end{equation}
with a small constant $\gamma$ depending on $\gamma_1$ and $\gamma_2$. 
We set $F_\infty=(\rho_\infty\chi_1,\chi_2/\rho_\infty)$ and apply our results for the linearized collision operator to $F-F_\infty$, where $F$ solves \eqref{nonlin}. 
Note that since $L(F_\infty)=0$ we have for evolution under the full equation \eqref{nonlin}
\[
  \partial_t\left(F-F_\infty\right) + \mathsf{T} \left(F-F_\infty\right) = \mathsf{L} \left(F-F_\infty\right)+Q(f,g)-\mathsf{L}F\, .
\]
Thus equation \eqref{coercH1Estim} becomes, including the derivative of the nonlinear interaction,
\begin{equation}\label{coercbilin}
\tfrac{\ud }{\ud t}\|F-F_\infty\|_{\mathcal{H}^1}^2\leq
-\tau\|F-F_\infty\|_{\mathcal{H}^1}^2+2\|F-F_\infty\|_{\mathcal{H}^1}\|Q(f,g)-L(F)\|_{\mathcal{H}^1}\,.
\end{equation}
Estimate \eqref{quadrpert} holds for velocity-derivatives straightforwardly. For spatial derivatives we use the bounds of Theorem~\ref{existence} and the multiplicative structure of the nonlinearity to derive
\begin{multline*}
\|\nabla_v\left(Q(f,g)-L(F)\right)\|\leq\\
\left\| \binom{(f-\rho_\infty \chi_1)}{(g-\chi_2/\rho_\infty)}\nabla_x\binom{(\rho_g - 1/\rho_\infty)}{(\rho_f - \rho_\infty)}\right\| +
\left\|\binom{(\rho_g - 1/\rho_\infty)}{(\rho_f - \rho_\infty)}\nabla_x\binom{(f-\rho_\infty \chi_1)}{(g-\chi_2/\rho_\infty)}\right\|\leq\\
\gamma \left\|\nabla_x\binom{\chi_1(\rho_g - 1/\rho_\infty)}{\chi_2(\rho_f - \rho_\infty)}\right\|+\gamma\left\|\nabla_x\binom{(f-\rho_\infty \chi_1)}{(g-\chi_2/\rho_\infty)}\right\|\leq \gamma\|F-F_\infty\|_{\mathcal{H}^1}\,,
\end{multline*}
where $\gamma$ is a small constant that changes form line to line and we used the fact that $\int\chi_1\ud v=1=\int\chi_2\ud v$ togehter with Chauchy--Schwarz in the last estimate.
We infer
\[
 \|Q(f,g)-L(F)\|_{\mathcal{H}^1}\leq \gamma \|F-F_\infty\|_{\mathcal{H}^1}
\]
and using this in \eqref{coercbilin} yields the exponential convergence as long as $\gamma_1$ and $\gamma_2$, and thus $\gamma$, are small enough.
\end{proof}

\subsection{Coercivity in a weighted $L^2$-space}\label{coercL2}
In this section we apply the abstract convergence theory of \cite{DMS}. This approach does not use derivatives to quantify the mixing effect of 
the transport but rather uses a modified $L^2$ entropy functional.
We again start with the linearized equation \eqref{lin2} in the abstract form 
\begin{equation*}
  \frac{\ud F}{\ud t} + \mathsf{T} F = \mathsf{L} F \,.
\end{equation*}  
It is easily seen that $\mathsf{L}$ is symmetric and $\mathsf{T}$ is skew-symmetric in the Hilbert space $\mathcal{H}$, defined as the
weighted $L^2$-space with the scalar product given in \eqref{def:scalar}. In this space the map $\Pi$, defined in \eqref{localequi}, is the orthogonal projection to the null space of $\mathsf{L}$. 

The approach of \cite{DMS} relies on the modified entropy functional
\[
  \mathsf{H}[F] := \frac{\|F\|^2}{2} + \delta \langle \mathsf{A} F,F \rangle \,,\quad\mbox{with } 
  \mathsf{A} := (1+(\mathsf{T}\Pi)^*\mathsf{T}\Pi)^{-1} (\mathsf{T}\Pi)^* \,,
\]
and with a small positive constant $\delta$. The time derivative of the modified entropy along solutions of \eqref{eq:lin-abstr} is
\begin{equation}\label{entropy-decay}
  \frac{\ud \mathsf{H}[F]}{\ud t} = \langle \mathsf{L} F, F\rangle - \delta \langle \mathsf{AT}\Pi F,F\rangle 
  - \delta \langle \mathsf{AT}(1-\Pi) F,F\rangle + \delta \langle \mathsf{TA} F,F\rangle + \delta \langle \mathsf{AL} F,F\rangle \,.
\end{equation}
The first term on the right hand side suggests that the microscopic coercivity estimate \eqref{microcoerc} is one of the necessary ingredients.
Since the operator $\mathsf{AT}\Pi$ can be interpreted as the application of the map $z\mapsto\frac{z}{1+z}$ to 
$(\mathsf{T}\Pi)^*\mathsf{T}\Pi$, the second condition called {\em macroscopic coercivity} is also plausible: There exists $\lambda_M>0$,
such that
\begin{equation}\label{maccoerc}\tag{P2}
 \|\mathsf{T}\Pi F \|^2\geq \lambda_M\|\Pi F\|^2\,.
\end{equation}
As a consequence of (P1) and (P2), the sum of the first two terms in the entropy dissipation above is coercive. It controls the remaining
three terms, if the operators appearing there are bounded and act only on the microscopic part $(1-\Pi)F$ of the distribution. For the operator
$\mathsf{A}$ this is guaranteed by the algebraic condition
\begin{equation}\label{parmac}
 \Pi\mathsf{T}\Pi=0\,,\tag{P3}
\end{equation}
called {\em parabolic macroscopic dynamics} (see Section 1.3). The final condition is the boundedness of the auxiliary operators: There
exists $C_M>0$, such that
\begin{equation}\label{aux-op}
 \|\mathsf{AT}(1-\Pi)F\| + \|\mathsf{AL}F\| \le C_M\|(1-\Pi)F\| \,.\tag{P4}
\end{equation}
Boundedness results of the same form for $\mathsf{A}$ and $\mathsf{TA}$ hold as a consequence of (P3) (see Lemma 1 of \cite{DMS}).
The former leads to coercivity of $\mathsf{H}[F]$ for $\delta$ small enough.

We formulate the abstract convergence result from \cite{DMS}:

\begin{theorem} Let $\mathrm{(P1)}$--\,\eqref{aux-op} hold and let $F$ be a solution of \eqref{eq:lin-abstr} subject to $F(0) = F_0 \in \mathcal{H}$. Then there
exist constants $\lambda, C>0$, such that
\[
  \|F(t)\| \le C e^{-\lambda t} \|F_0\| \,.
\]
\end{theorem}

For our problem it remains to verify (P2)--(P4).
A straightforward calculation shows that (P2) is equivalent to 
\[
 \int_{\mathbb{T}^3} \nabla_x u_F^{tr}D_0\nabla_x u_F \,\ud x
  \ge \lambda_M \int_{\mathbb{T}^3} u_F^2\ud x \,,
\]
with $u_F=\rho_f-\rho_g$ satisfying $\int_{\mathbb{T}^3} u_F\,\ud x = 0$ by \eqref{mass-zero}. Thus, by the positive definiteness of
$D_0 = (\rho_\infty^2 + 1)^{-1}(\rho_\infty^2 D_1 + D_2)$, (P2) is a consequence of the Poincar{\'e} inequality on $\mathbb{T}^3$. Since
\[
  \mathsf{T}\Pi F =\frac{v\cdot\nabla_x u_F}{\rho_\infty^2 + 1} \binom{\rho_\infty \chi_1}{-\chi_2} \,,
\]
and since the application of $\Pi$ involves an integration with respect to $v$, the assumptions \eqref{ass:chi} imply \eqref{parmac}.
By \eqref{lincolbd} the linearized collision operator $\mathsf{L}$ is bounded. For the verification of (P4) it is thus sufficient to prove
the boundedness of $\mathsf{AT}$ or, equivalently, of its adjoint 
$(\mathsf{AT})^* = -\mathsf{TA}^* = -\mathsf{T}^2 \Pi(1+(\mathsf{T}\Pi)^*\mathsf{T}\Pi)^{-1}$. The equation 
$G = (1+(\mathsf{T}\Pi)^*\mathsf{T}\Pi)^{-1} F$ implies
\[
  u_G - \nabla_x\cdot (D_0\nabla_x u_G) = u_F \,.
\]
The norm $\|\mathsf{T}^2 \Pi G\|$ is equivalent to the $L^2(\mathbb{T}^3)$-norm of $\nabla_x^2 u_G$, whose boundedness in terms
of the $L^2(\mathbb{T}^3)$-norm of $u_F$ (and therefore in terms of $\|F\|$) is a consequence of elliptic regularity. This proves (P4)
and completes the proof of exponential decay to equilibrium for the linearized problem.

\begin{theorem}
Let \eqref{ass:chi} hold, let $f_0 \in L^2(\ud v/\chi_1)$, $g_0 \in L^2(\ud v/\chi_2)$, and let 
$\int_{\mathbb{T}^3}\int_{\mathbb{R}^3}(f_0-g_0)\ud v\,\ud x = 0$. Then the solution of \eqref{lin2} subject to $f(t=0)=f_0$, $g(t=0)=g_0$,
satisfies
\[
  \|f(\cdot,\cdot,t)\|_{L^2(\ud v/\chi_1)} + \|g(\cdot,\cdot,t)\|_{L^2(\ud v/\chi_2)}  \le C e^{-\lambda t} \,,
\]
with positive constants $C$ and $\lambda$.
\end{theorem}

Since the maximum principle estimates of Theorem \ref{existence} already imply a stability (but not asymptotic stability) result for the nonlinear
problem, the decay result can be extended to a local result for the nonlinear case by the same method.

\begin{theorem}
Let \eqref{ass:chi} hold and let $f_0$ and $g_0$ satisfy the assumptions of Theorem \ref{existence} with $\gamma_1$ and $\gamma_2$ small
enough. Then the solution of the initial value problem \eqref{nonlin}, \eqref{IC} satisfies
\[
  \|f(\cdot,\cdot,t)-\rho_\infty \chi_1\|_{L^2(\ud v/\chi_1)} + \|g(\cdot,\cdot,t)-\chi_2/\rho_\infty\|_{L^2(\ud v/\chi_2)}  \le C e^{-\lambda t} \,,
\]
with positive constants $C$ and $\lambda$.
\end{theorem}

\begin{proof}
We start by writing the problem in terms of the unknown $F = (f-\rho_\infty \chi_1,g-\chi_2/\rho_\infty)^T$. Then we proceed as above
producing the entropy decay relation \eqref{entropy-decay}, however with $\mathsf{L} F$ replaced by 
\[
  Q(f,g) = \binom{\chi_1 - \rho_g f}{\chi_2 - \rho_f g}\,.
\]
Since, by Theorem \ref{existence}, $\rho_f$ is close to $\rho_\infty$ and $\rho_g$ to $1/\rho_\infty$ we have again (cf. \eqref{quadrpert})
\[
  \| Q(f,g) - \mathsf{L} F\| \le \gamma \|F\|\,,
\]
with a small constant $\gamma$. Therefore the entropy dissipation for the nonlinear equation is a small perturbation of the entropy
dissipation of the linearized problem, which does not destroy its coercivity.
\end{proof}

\section{Rigorous macroscopic limit}\label{RigHy}

Our next goal is to validate the macroscopic limit carried out formally in Section \ref{formal-limit}. We start from the rescaled system 
\eqref{rescaled}, subject to initial conditions, where the data satisfy the assumptions of Theorem \ref{existence}. The results of Theorem
\ref{existence} remain valid with the $\epsilon$-independent $L^\infty$-bounds. We start by exploiting entropy decay.

\begin{lemma}\label{uniform-est}
Let the assumptions of Theorem \ref{existence} hold, let $(f,g)$ be the solution of \eqref{IC}, \eqref{rescaled} for $\epsilon>0$, and define the
micro-macro decompositions $f(x,v,t) = \rho_f(x,t)\chi_1(v) + \epsilon f^\perp(x,v,t)$, $g(x,v,t) = \rho_g(x,t)\chi_2(v) + \epsilon g^\perp(x,v,t)$.  
Then $f^\perp$ and $g^\perp$ are bounded uniformly in $\epsilon$ in $L^2(\mathbb{T}^3\times \mathbb{R}^3\times (0,\infty), \ud x\,\ud v\,\ud t/\chi_j)$, $j=1,2$, respectively, and $\sqrt{\rho_f \rho_g} - 1 = O(\epsilon)$ in $L^2(\mathbb{T}^3\times (0,\infty), \ud x\,\ud t)$.
\end{lemma}

\begin{proof}
The proof is based on the entropy decay relation
\begin{equation*}
 \frac{\varepsilon^2}{2}\frac{\ud H(f,g)}{\ud t}=
 \int_{\mathbb{T}^3}\int_{\mathbb{R}^3}\int_{\mathbb{R}^3}\chi_1\chi_2'\left(1-\frac{fg'}{\chi_1\chi_2'}\right)\ln\frac{fg'}{\chi_1\chi_2'}\ud v'\ud v\ud x\,.
\end{equation*}
Since the entropy is uniformly bounded in $\epsilon$ and $t$,  using $\left(\sqrt{a}-1\right)^2\leq \tfrac{1}{4}(a-1)\ln a$, we derive
\[
\int_{0}^{\infty}\int_{\mathbb{T}^3}\int_{\mathbb{R}^3}\int_{\mathbb{R}^3}\chi_1\chi_2'\left(\sqrt{\frac{fg'}{\chi_1\chi_2'}}-1\right)^2\ud v'\ud v\ud x\ud t = O(\varepsilon^2) \,.
\]
Using the micro-macro decomposition and expanding the square we find
\begin{multline}\label{hydr:entr}
 I(t) :=\int_{\mathbb{T}^3}\int_{\mathbb{R}^3}\int_{\mathbb{R}^3}\chi_1\chi_2'\left(\sqrt{\frac{fg'}{\chi_1\chi_2'}}-1\right)^2\ud v'\ud v\ud x= \int_{\mathbb{T}^3}\left(\rho_f\rho_g+1\right)\ud x\\
- 2\int_{\mathbb{T}^3}\int_{\mathbb{R}^3}\int_{\mathbb{R}^3} \chi_1\chi_2'\sqrt{\rho_f\rho_g} \sqrt{\left(1+\frac{\varepsilon f^{\perp}}{\rho_f\chi_1}\right)\left(1+\frac{\varepsilon g^{\perp'}}{\rho_g\chi'_2}\right)}\ud v'\ud v\ud x \,.
\end{multline}
Now we use the identity
\[
  \sqrt{(1+\epsilon a)(1+\epsilon b)} = 1+\frac{\epsilon a}{2} + \frac{\epsilon b}{2} 
  - \frac{\epsilon^2(a-b)^2}{4(\sqrt{(1+\epsilon a)(1+\epsilon b)} + 1 + \epsilon a/2 + \epsilon b/2)}
\]
with $a = f^\perp/(\rho_f\chi_1)$, $b = g^{\perp'}/(\rho_g\chi_2')$. Since $1+\epsilon a = f/(\rho_f\chi_1)$, 
$1+\epsilon b = g'/(\rho_g\chi_2')$, the estimates from Theorem \ref{existence} can be used to obtain 
$1+\epsilon a, 1+\epsilon b \le \frac{\rho_\infty+\gamma_2}{\rho_\infty-\gamma_1}$, with the consequence
\[
  \sqrt{(1+\epsilon a)(1+\epsilon b)} \le 1+\frac{\epsilon a}{2} + \frac{\epsilon b}{2} 
  - \frac{\epsilon^2(\rho_\infty-\gamma_1)(a-b)^2}{8(\rho_\infty + \gamma_2)}
\]
Using this in \eqref{hydr:entr}, we obtain
\[
  O(\epsilon^2) = I(t) \ge \int_{\mathbb{T}^3} (\sqrt{\rho_f\rho_g} - 1)^2 \ud x + \frac{\epsilon^2(\rho_\infty-\gamma_1)}{4(\rho_\infty + \gamma_2)^3}
  \int_{\mathbb{T}^3}\int_{\mathbb{R}^3} \left( \frac{f^{\perp2}}{\chi_1} + \frac{g^{\perp2}}{\chi_2}\right)\ud v\,\ud x \,,
\]
completing the proof.
\end{proof}

With this basis we now follow the procedure of \cite{PouSch}. In particular, we use an averaging lemma, which can be proved similarly to 
Lemma 3.2 in \cite{PouSch}. We give the short proof for completeness.

\begin{lemma}\label{averaging}
Let $\chi$ satisfy \eqref{ass:chi}, let $f$ and $h$ lie in subsets of $L^2(\mathbb{T}^3\times\mathbb{R}^3\times\mathbb{R}, \ud x\,\ud v\,\ud t/\chi(v))$, uniformly bounded in terms of the small parameter $\epsilon$, and let $\epsilon\,\partial_t f + v\cdot\nabla_x f = h$. 
Then $\rho_f$ is bounded uniformly in $\epsilon$ in $L^2(\mathbb{R}; H^{\theta/(2+\theta)}(\mathbb{R}^3))$.
\end{lemma}

\begin{proof}
We represent $f$ by the Fourier transform with respect to $t$ and by the Fourier series with respect to $x$:
\[
  f(x,v,t) = \sum_{\xi\in \mathbb{T}^{3*}} \int_{\mathbb{R}} \hat f(\xi,v,\tau) e^{i(t\tau + x\cdot\xi)} d\tau \,,
\]
with the lattice $\mathbb{T}^{3*}$ dual to the torus $\mathbb{T}^3$, implying
\[
  z\hat f = -i\hat h \,,\qquad\mbox{with } z = \epsilon\tau + v\cdot\xi \,.
\]
For each $\lambda>0$, we introduce a smooth, nonnegative real function $\psi_\lambda(z)\le 1$, satisfying $\psi_\lambda(z) = 0$ for 
$|z|\le\lambda$ and $\psi_\lambda(z) = 1$ for $|z|\ge 2\lambda$. Now we estimate, using \eqref{ass:chi},
\begin{align*}
  |\hat\rho_f| &\le \left| \int_{\mathbb{R}^3} \frac{\psi_\lambda}{z} \hat h\,\ud v\right| 
       + \left| \int_{\mathbb{R}^3} (1 - \psi_\lambda) \hat f\,\ud v\right| \\
  &\le \left( \int_{\mathbb{R}^3} \frac{\psi_\lambda^2}{z^2} \chi\,\ud v\right)^{1/2} \|\hat h\|_{L^2(\ud v/\chi)} + 
       \left( \int_{\mathbb{R}^3} (1 - \psi_\lambda)^2 \chi\,\ud v\right)^{1/2} \|\hat f\|_{L^2(\ud v/\chi)} \\
  &\le  \frac{1}{\lambda} \|\hat h\|_{L^2(\ud v/\chi)} + \sqrt{C} \left( \frac{2\lambda}{|\xi|}\right)^{\theta/2} \|\hat f\|_{L^2(\ud v/\chi)} \,.
\end{align*}
With the optimal choice $\lambda = |\xi|^{\theta/(2+\theta)}$, we obtain
\[
  |\xi|^{\theta/(2+\theta)} |\hat\rho_f| \le c\left( \|\hat f\|_{L^2(\ud v/\chi)} + \|\hat h\|_{L^2(\ud v/\chi)} \right) \,,
\]
completing the proof.
\end{proof}

\begin{theorem}
Let the assumptions of Theorem \ref{existence} hold. Then as $\epsilon\to 0$ the solution $(f,g)$ of \eqref{IC}, \eqref{rescaled} converges to
$(\rho\chi_1, \chi_2/\rho)$ in $L_{loc}^2(\mathbb{T}^3\times \mathbb{R}^3\times (0,\infty), \ud x\,\ud v\,\ud t/\chi_1) \times
L_{loc}^2(\mathbb{T}^3\times \mathbb{R}^3\times (0,\infty), \ud x\,\ud v\,\ud t/\chi_2)$, when restricting to subsequences, where
$\rho\in L^\infty(\mathbb{T}^3\times (0,\infty))$ satisfies $\rho_\infty - \gamma_1 \le \rho \le \rho_\infty + \gamma_2$. Furthermore
there exist $J_1,J_2 \in L^2(\mathbb{T}^3\times (0,\infty))^3$ such that
\[
  \partial_t \left( \rho - \frac{1}{\rho}\right) + \nabla_x\cdot (J_1 - J_2) = 0 \,,\quad \frac{1}{\rho} J_1 = -D_1\nabla_x\rho \,,\quad
  \rho J_2 = -D_2\nabla_x \frac{1}{\rho} \,,
\]
hold in the sense of distributions.
\end{theorem}

\begin{proof}
Because of the boundedness of $\rho_f$ and $\rho_g$ and of Lemma \ref{uniform-est}, $f$ and the function
\[
  h := \frac{\chi_1 - \rho_g f}{\epsilon} = \chi_1 (1 + \sqrt{\rho_f \rho_g})\frac{1 - \sqrt{\rho_f \rho_g}}{\epsilon} - \rho_g f^\perp
\]
satisfy the assumptions of Lemma \ref{averaging} with $\chi=\chi_1$ (after even extension to $t<0$). With an analogous argument for $g$
we obtain $\rho_f, \rho_g\in L^2((0,\infty); H^{\theta/(2+\theta)}(\mathbb{R}^3))$ uniformly in $\epsilon$.

The conservation law
\begin{equation}\label{cons-law}
  \partial_t (\rho_f - \rho_g) + \nabla_x\cdot \left( \int_{\mathbb{R}^3} v(f^\perp - g^\perp)\ud v\right) = 0 \,,
\end{equation}
the observation
\begin{equation}\label{flux-bound}
  \left| \int_{\mathbb{R}^3}v (f^\perp - g^\perp) \ud v\right| \le \sqrt{\mbox{tr}D_1} \,\|f^\perp\|_{L^2(\ud v/\chi_1)} 
    + \sqrt{\mbox{tr}D_2} \,\|g^\perp\|_{L^2(\ud v/\chi_2)}\,,
\end{equation}
and Lemma \ref{uniform-est} imply $\rho_f-\rho_g \in H^1((0,\infty); H^{-1}(\mathbb{R}^3))$ which, after interpolation with the averaging result
gives 
\[
  \rho_f-\rho_g \in H^{\frac{\theta}{2(1+\theta)}}((0,\infty)\times\mathbb{R}^3) \quad\mbox{uniformly in } \epsilon \,.
\]
As a consequence, for each $0\le a < b$ and compact $K\subset\mathbb{R}^3$, a subsequence of $\rho_f-\rho_g$ converges strongly 
in $L^2((a,b)\times K)$ as $\epsilon\to 0$. Since the same is true for $\sqrt{\rho_f \rho_g}\to 1$, it also holds for $\rho_f$ and $\rho_g$
individually as a consequence of the $L^\infty$ bounds. Another application of Lemma \ref{uniform-est} completes the proof of the convergence
statement.

For the derivation of the limiting problem, we pass to the limit in \eqref{cons-law} in the distributional sense, denoting the weak limits
of $\int_{\mathbb{R}^3} v f^\perp \ud v$ and $\int_{\mathbb{R}^3} v g^\perp \ud v$, which exist because of \eqref{flux-bound}, by
$J_1$ and $J_2$, respectively. Now we multiply the equation for $f$ by $v/\epsilon$ and integrate with respect to $v$ obtaining
\[
  \epsilon \partial_t \int_{\mathbb{R}^3} v f^\perp \ud v + \nabla_x\cdot \int_{\mathbb{R}^3} v\otimes v f \ud v 
    = - \rho_g \int_{\mathbb{R}^3} v f^\perp \ud v \,.
\]
By the uniform-in-$\epsilon$ boundedness of $\int_{\mathbb{R}^3} v\otimes v f \ud v$ (consequence of Theorem \ref{existence} and 
\eqref{ass:chi}) and by the strong convergence of $\rho_g$, we can pass to the limit, leading to the desired equation for $J_1$. For $J_2$
we proceed analogously.
\end{proof}
%

\bigskip
\begin{flushleft} \signln \end{flushleft} 
\begin{flushleft} \signcs \end{flushleft} 
\end{document}